\documentclass[11pt, reqno]{amsart}
\usepackage{amsaddr}
\usepackage{latexsym,amsfonts}
\usepackage{amssymb,amsthm,upref,amscd}
\usepackage[T1]{fontenc}
\usepackage{times}
\usepackage{amscd}
\usepackage{enumerate}
\usepackage{mathrsfs}
\usepackage{amsmath}
\usepackage{color}
\usepackage{soul}
\usepackage{indentfirst}
\usepackage{comment}
\PassOptionsToPackage{reqno}{amsmath}
\usepackage{amsmath,amsfonts,amsthm,amssymb,bbm}
\usepackage{graphicx,color,dsfont}
\usepackage{enumitem}
\usepackage{fourier}

\newtheorem{thm}{Theorem}[section]

\newtheorem{lem}{Lemma}[section]

\newtheorem{rem}{Remark}[section]
\theoremstyle{notation}

\newcommand{\R}{\mathbb{R}}
\newcommand{\C}{\mathbb{C}}
\numberwithin{equation}{section}

\makeatletter \@addtoreset{equation}{section} \makeatother

\usepackage[textwidth=138mm,
textheight=215mm,
left=30mm,
right=30mm,
top=25.4mm,
bottom=25.4mm,
headheight=2.17cm,
headsep=4mm,
footskip=12mm,
heightrounded,]{geometry}

\usepackage[colorlinks=true,urlcolor=blue,
citecolor=black,linkcolor=black,linktocpage,pdfpagelabels,
bookmarksnumbered,bookmarksopen]{hyperref}
\newcounter{const}
\author[{}]{Tianxiang Gou$^{\MakeLowercase a}$, Vicen\c tiu D. R\u adulescu$^{\MakeLowercase{b,c}}$, Zhitao Zhang$^{\MakeLowercase{d,e}}$}

\address[T. Gou]{%
\centerline{$^a$School of Mathematics and Statistics, Xi’an Jiaotong University,}
\centerline{710049 Xi’an, Shaanxi, China}}

\address[V.D. R\u adulescu]{%
	\centerline{$^b$Faculty of Applied Mathematics, AGH University of Krak\'ow,} 
	\centerline{30-059 Krak\'ow, Poland}
	\centerline{$^c$Department of Mathematics, University of Craiova,} 
	\centerline{200585 Craiova, Romania}}

\address[Z. Zhang]{%
\centerline{$^d$Academy of Mathematics and Systems Science, Chinese Academy of Sciences,}
\centerline{100190 Beijing, China}
\centerline{$^e$School of Mathematical Sciences, University of Chinese Academy of Sciences,}
\centerline{100149 Beijing, China}}



\email{tianxiang.gou@xjtu.edu.cn; radulescu@inf.ucv.ro; zzt@math.ac.cn}
\title[Blow-up for Fractional NLS]{Blow-up of cylindrically symmetric solutions for Fractional NLS}

\begin{document}

\begin{abstract}
In this paper, we consider blow-up of solutions to the Cauchy problem for the following fractional  NLS,
$$
\textnormal{i} \, \partial_t u=(-\Delta)^s u-|u|^{2 \sigma} u \quad \text{in} \,\, \R \times \R^N,
$$
where $N \geq 2$, $1/2 <s<1$ and $0<\sigma<2s/(N-2s)$. In the mass critical and supercritical cases, we establish a criterion for blow-up of solutions to the problem for cylindrically symmetric data. The results extend the known ones with respect to blow-up of solutions to the problem for radially symmetric data in \cite{BHL}.
\end{abstract}
\maketitle
\thispagestyle{empty}

\section{Introduction}

In this paper, we are concerned with blow-up of cylindrically symmetric solutions to the following fractional NLS,
\begin{align} \label{equ}
\left\{
\begin{aligned}
&\textnormal{i} \, \partial_t u=(-\Delta)^s u-|u|^{2 \sigma} u, \\
& u(0)=u_0 \in H^s(\R^N),
\end{aligned}
\right.
\end{align}
where $N \geq 2$, $1/2 <s<1$ and $0<\sigma<2s/(N-2s)$. The fractional Laplacian $(-\Delta)^s$ is characterized as $\mathcal{F}((-\Delta)^{s}u)(\xi)=|\xi|^{2s} \mathcal{F}(u)(\xi)$ for $\xi \in \R^N$, where $\mathcal{F}$ denotes the Fourier transform. Equation \eqref{equ} that was introduced by Laskin in \cite{L1, L2} can be seen as a canonical model for a nonlocal dispersive PDE with focusing nonlinearity. Evolution equations with nonlocal dispersion as \eqref{equ} naturally arise in various physical settings, such as in the continuum limit of discrete models with long-range interactions \cite{KLS} and in the description of Boson stars as well as in water wave dynamics \cite{Cai, ES, MMT}.

When $s=1$, the early study of the existence of finite time blow-up solutions to \eqref{equ} for initial data with finite variance is due to Glassy in \cite{Gl}. The result was later extended by Holmer and Roudenko in \cite{HR} and Ogawa and Tsutsumi in \cite{OT} to radially symmetric initial data with infinite variance. 
While $0<s<1$, despite that \eqref{equ} bears a strong resemblance to the classical NLS, the existence of blow-up solutions to \eqref{equ} was open for a long time until the work of Boulenger et al. \cite{BHL}. In \cite{BHL}, they proved a general criterion for blow-up of solutions to \eqref{equ} with radially symmetric data. 
Nevertheless, the consideration of blow-up solutions to \eqref{equ} with non-radially symmetric data left open so far. Inspired by the aforementioned works, the first aim of the present paper is to investigate blow-up of solutions to \eqref{equ} with initial data belonging to $\Sigma_N$ defined by
$$
\Sigma_N :=\left\{ u \in H^s(\R^N) : u(y, x_N) =u(|y|, x_N), x_N u \in L^2(\R^N)\right\},
$$
where $x=(y, x_N) \in \R^N$ and $y=(x_1, \cdots, x_{N-1}) \in \R^{N-1}$. And we derive the existence of finite time blow-up of solutions to \eqref{equ} with initial data belonging to $\Sigma_N$ in the mass supercritical case $\sigma>{2s}/{N}$, see Theorem \ref{thm1}. The second aim of the paper is to establish blow-up of solutions to \eqref{equ} with initial data belonging to $\Sigma$ defined by
$$
\Sigma :=\left\{ u \in H^s(\R^N) : u(y, x_N) =u(|y|, x_N)\right\}.
$$
And we obtain the existence of finite time blow-up of solutions to \eqref{equ} with initial data belonging to $\Sigma$ for $\sigma>{2s}/{(N-1)}$, see Theorem \ref{thm11}. It is worth quoting \cite{BF, BFG, DF, F, Gou}, where blow-up of cylindrically symmetric solutions to the local NLS with initial data in $\Sigma_N$ was well-considered. And these results can be regarded as extensions of the ones in the seminal work due to Martel \cite{M}. Since the problem under our consideration is nonlocal, then essential arguments we adapt here are greatly different from the ones used to deal with the local NLS.

Let us now mention the work of Hong and Sire \cite{HS}, where the local well-posedness of solutions to \eqref{equ} in $H^s(\R^N)$ was investigated. Problem \eqref{equ} satisfies the conservation of the mass and the energy given respectively by
$$
M[u]:=\int_{\R^N} |u|^2 \,dx,
$$
$$
E[u]:=\frac 12 \int_{\R^N} |(-\Delta)^{s/2} u|^2 \,dx-\frac{1}{2 \sigma +2} \int_{\R^N} |u|^{2 \sigma+2} \,dx.
$$
For further clarifications, we shall fix some notations. Let us define 
$$
s_c:=\frac N 2 -\frac {s}{\sigma}.
$$
We refer to the cases $s_c<0$, $s_c=0$ and $s_c>0$ as mass subcritical, critical and supercritical, respectively. The end case $s_c=s$ is energy critical. Note that the cases $s_c=0$ and $s_c=s$ correspond to the exponents $\sigma=2s/N$ and $\sigma=2s/(N-2s)$, respectively. For $1 \leq p <\infty$ and $d \geq 1$, we denote by $L^q(\R^{d})$ the usual Lebesgue space with the norm
$$
\|u\|_p:=\left( \int_{\R^d} |u|^p \,dx \right)^{\frac 1p}, \quad 1 \leq p<\infty, \quad \|u\|_{\infty}:= \mathop{\mbox{ess sup}}_{x \in \R^d} |u|.
$$
The Sobolev space $H^s(\R^d)$ is equipped with the standard norm
$$
\|u\|:=\|(-\Delta)^{s/2}u\|_2 +\|u\|_2.
$$
In addition, we denote by $Q \in H^s(\R^N)$ the ground state to the following fractional nonlinear elliptic equation,
$$
(-\Delta)^s Q +Q -|Q|^{2\sigma} Q=0 \quad \text{in} \,\, \R^N.
$$
The uniqueness of ground states was recently revealed in \cite{FL, FLS}. Throughout the paper, we shall write $X \lesssim Y$ to denote that $X \leq C Y$ for some irrelevant constant $C>0$.

The main results of the present paper read as follows and they provide blow-up criteria for solutions of problem \eqref{equ} with cylindrically symmetric data. 

\begin{thm} \label{thm1} \textnormal{(Blowup for Mass-Supercritical Case)}
Let $N \geq 3$, $1/2<s<1$ and $0<s_c \leq s$ with $0<\sigma \leq s$. Suppose that $u_0 \in \Sigma_N$ satisfies that either $E[u_0]<0$ or $E[u_0] \geq 0$ such that
\begin{align} \label{c1}
E[u_0]^{s_c}M[u_0]^{s-s_c}<E[Q]^{s_c}M[Q]^{s-s_c},
\end{align}
\begin{align} \label{c2}
\|(-\Delta)^{s/2} u_0\|_2^2 \|u_0\|_2^{s-s_c}>\|(-\Delta)^{s/2} Q\|_2^2 \|Q\|_2^{s-s_c}.
\end{align}
Then the solution $u \in C([0, T), H^s(\R^N))$ to \eqref{equ} with initial datum $u_0$ blows up in finite time, i.e. $0<T<+\infty$.
\end{thm}

\begin{rem}
When $s_c=s$, by \cite[Appendix B]{BHL}, then $Q \not\in L^2(\R^N)$. In this situation, the conditions \eqref{c1} and \eqref{c2} become the following ones,
$$
E[u_0]<E[Q], \quad \|(-\Delta)^{s/2} u_0\|_2>\|(-\Delta)^{s/2} Q\|_2.
$$

\end{rem}


\begin{thm} \label{thm2} \textnormal{(Blowup for Mass-Critical Case)}
Let $N \geq 3$, $1/2<s<1$ and $s_c=0$ with $0<\sigma \leq s$. Suppose that $u_0 \in \Sigma_N$ satisfies that $E(u_0)<0$. Then the solution $u \in C([0, T), H^s(\R^N))$ to \eqref{equ} either blows up in finite time or blows up in infinite time 
such that
$$
\|(-\Delta)^{s/2} u(t)\|_2 \geq C t^s \quad \text{for any} \,\, t \geq t_0,
$$
where $C>0$ and $t_0=t_0>0$ are constants depending only on $u_0, s$ and $N$.
\end{thm}

\begin{rem}
The assumption that $0<\sigma \leq s$ is technical. It is unknown whether Theorems \ref{thm1} and \ref{thm2} remain hold for $\sigma>s$.
\end{rem}

To prove Theorems \ref{thm1} and \ref{thm2}, the crucial arguments lie in establishing localized virial estimates \eqref{rv0} and \eqref{rv} for cylindrically symmetric solutions to \eqref{equ}, where \eqref{rv} is a refine version of \eqref{rv0} used to discuss blow-up of the solutions to \eqref{equ} for $s_c=0$. First we need to introduce a localized virial quantity $\mathcal{M}_{\varphi_R}[u]$ defined by \eqref{defm}, where $\varphi_R$ defined by \eqref{defv} is a cylindrically symmetric function for $R>0$. Then, adapting \cite[Lemma 2.1]{BHL}, we can derive the virial identity \eqref{m1}. At this stage, to get the desired conclusions, one of the key arguments is actually to estimate the following term,
\begin{align} \label{non}
\int_{\R} \int_{|y| \geq R} |u|^{2 \sigma +2} \,dydx_N.
\end{align}
To do this, when $s=1$, one needs to make use of the following two crucial ingredients jointly with the classical Gagliardo-Nirenberg's inequality in $H^1(\R)$ and the radial Sobolev's inequality in $H^1(\R^{N-1})$,
\begin{align} \label{ineq1}
\sup_{x_N \in \R} \int_{\R^{N-1}} |u|^2 \,dy \lesssim \left(\int_{\R^{N-1}} |\partial_{x_N} u|^2 \,dy\right)^{\frac 12},
\end{align}
\begin{align} \label{ineq2}
\partial_{x_N} \left(\int_{\R^{N-1}} |u|^2 \,dy\right)^{\frac 12} \leq \left(\int_{\R^{N-1}} |\partial_{x_N} u|^2 \,dy\right)^{\frac 12},
\end{align}
see for example \cite{BFG, DF, M}. However, to our knowledge, it seems rather difficult to generalize the estimates \eqref{ineq1} and \eqref{ineq2} to the nonlocal ones. Therefore, we cannot follow the strategies in \cite{BFG, DF, M} to handle the term \eqref{non}. In fact, when $0<s<1$, by tactfully employing Sobolev's inequality in $W^{s, 1}(\R)$ and certain fractional chain rule, we then have that
$$
\int_{\R}\|u\|^{\frac{2}{1-\sigma}}_{L^2_y} \, d x_N \lesssim \left(\int_{\R^{N-1}} \left\|(-\partial_{x_Nx_N})^{{s}/{2}} u \right\|_{L^2_{x_N}}^2 \,dy \right)^{\frac{\sigma}{2s(1-\sigma)}}, \quad 0<\sigma \leq s.
$$
With this at hand, we are now able to estimate the term \eqref{non}. This then completes the proofs.

It would be interesting to investigate blow-up of solutions to \eqref{equ} for cylindrically symmetric initial data belonging in $\Sigma$ without the restriction that $x_N u_0 \in L^2(\R^N)$. In this respect, we have the following result.

\begin{thm} \label{thm11} \textnormal{(Blowup for Mass-Supercritical Case Revisited)}
Let $N \geq 4$, $1/2<s<1$ and $0<s_c \leq s$ with ${2s}/{(N-1)}<\sigma \leq s$. Suppose that $u_0 \in \Sigma$ satisfies that $E[u_0]<0$.
Then the solution $u \in C([0, T), H^s(\R^N))$ to \eqref{equ} with initial datum $u_0$ blows up in finite time, i.e. $0<T<+\infty$.
\end{thm}

To establish Theorem \ref{thm11}, we need to introduce a new localized virial quantity $\mathcal{M}_{\psi_R}[u]$ defined by \eqref{defm1}. Following closely the proof of Lemma \ref{v}, one can get localized virial estimate \eqref{rv1} for cylindrically symmetric solutions to \eqref{equ}. This then implies the desired conclusion.


\section{Proofs of main results}

In this section, we are going to prove Theorems \ref{thm1} and \ref{thm2}. Let us first introduce a localized virial quantity. Let $\psi : \R^{N-1} \to \R$ be a radially symmetric and smooth function such that 
\begin{align*}
\psi(r):=\left\{
\begin{aligned}
&\frac{r^2}{2} &\quad r \leq 1,\\
&const. & \quad r \geq 10,
\end{aligned}
\right.
\qquad \psi''(r) \leq 1, \quad r \geq 0.
\end{align*}
Let $\psi_R : \R^{N-1} \to \R$ be a radially symmetric function defined by
\begin{align} \label{defp}
\psi_R(r):= R^2 \varphi\left(\frac{r}{R}\right), \quad R>0.
\end{align}
It is straightforward to verify that
\begin{align} \label{property}
1-\psi_R''(r) \geq 0, \quad 1-\frac{\psi'_R(r)}{r} \geq 0, \quad N-1-\Delta \psi_R(r) \geq 0, \quad r \geq 0.
\end{align}
Let $\varphi: \R^N=\R^{N-1} \times \R \to \R$ be a smooth function defined by
\begin{align} \label{defv}
\varphi_R(x):=\psi_R(r) + \frac{x_N^2}{2}, \quad  x=(y, x_N) \in \R^{N-1} \times \R, \,\, r=|y|.
\end{align}
Now we introduce the localized virial quantity as
\begin{align} \label{defm}
\mathcal{M}_{\varphi_R}[u]:=2 \textnormal{Im} \int_{\R^N} \overline{u} \left(\nabla \varphi_R \cdot \nabla u \right) \,dx.
\end{align}
It is simple see that $\mathcal{M}_{\varphi_R}[u]$ is well-defined for any $u \in \Sigma_N$. 

For convenience, we shall give the well-known fractional radial Sobolve's inequality in \cite{CO}. For every radial function $f \in H^s(\R^{N-1})$ with $N \geq 3$, then
\begin{align} \label{st}
|y|^{\frac{N-2}{2}} |f(y)| \leq C(N, s) \|(-\Delta)^{s/2} f\|_{2}^{\frac{1}{2s}} \|f\|_2^{1-\frac{1}{2s}}, \quad  y\neq 0.
\end{align}
Also we shall present the well-known Gagliardo-Nirenberge's inequality in \cite{Park}. For any $f \in H^1(\R)$ and $p>2$, then
\begin{align} \label{gn}
\|f\|_{p} \leq C(s, p) \|(-\Delta)^{s/2} f\|_2^{\alpha} \|f\|_2^{1-\alpha}, \quad \alpha=\frac{p-2}{2ps}.
\end{align}
Let $f : \R^{N-1} \to \C$ be a radially symmetric and smooth function, then 
\begin{align}\label{d}
\partial_{kl}^2 f=\left(\delta_{kl}-\frac{x_kx_l}{r^2}\right) \frac{\partial_r f}{r} +\frac{x_kx_l}{r^2} \partial^2_{rr}f.
\end{align} 

In the following, we are going to estimate the evolution of $\mathcal{M}_{\varphi_R}[u(t)]$ along time, which is the key to establish Theorems \ref{thm1} and \ref{thm2}.

\begin{lem} \label{v}
Let $N \geq 3$, $1/2<s<1$ and $0<\sigma \leq s$. Suppose that $u \in C([0, T); H^s(\R^N))$ is the solution to \eqref{equ} with initial datum $u_0 \in \Sigma_N$. Then, for any $t \in [0, T)$, there holds that
\begin{align} \label{rv0}
\begin{split}
\frac{d}{dt}\mathcal{M}_{\varphi_R}[u(t)] &\leq 4s \int_{\R^N}|(-\Delta)^{s/2} u(t)|^2 \,dx -\frac{2 \sigma N}{\sigma+1} \int_{\R^N}  |u(t)|^{2 \sigma +2} \,dx  \\
& \quad + C \left( R^{-2s} +\left(1+R^{-\sigma(N-2)} \int_{\R^N} |(-\Delta )^{{s}/{2}} u(t)|^2\,dx\right) \right)\\
&=4\sigma N E(u_0)-2(\sigma N -2s)\int_{\R^N} |(-\Delta )^{{s}/{2}} u(t)|^2\,dx \\
& \quad + C \left( R^{-2s} +\left(1+R^{-\sigma(N-2)} \int_{\R^N} |(-\Delta )^{{s}/{2}} u(t)|^2\,dx\right) \right).
\end{split}
\end{align}
\end{lem}
\begin{proof}
Define 
$$
u_m(t, x):=c_s \frac{1}{-\Delta +m} u(t, x)=c_s \mathcal{F}^{-1} \left(\frac{\mathcal{F}(u)(t, \xi)}{|\xi|^2+m}\right), \quad m>0,
$$
where
$$
c_s:=\sqrt{\frac{\sin \pi s}{\pi}}.
$$
It follows from \cite[Lemma 2.1]{BHL} that
\begin{align} \label{m1}
\begin{split} 
\frac{d}{dt} \mathcal{M}_{\varphi_R}[u(t)]=&\int_0^{\infty} \int_{\R^N} m^s \left(4\sum_{k, l=1}^N \overline{\partial_k u_m} \left(\partial_{kl}^2 \varphi_R\right) \partial_l u_m-\left(\Delta^2 \varphi_R\right) |u_m|^2 \right) \,dxdm \\
& -\frac{2\sigma}{\sigma+1} \int_{\R^N} \left(\Delta \varphi_R \right) |u|^{2\sigma+2} \,dx.
\end{split}
\end{align}
Let us start with treating the first term in the right hand side of \eqref{m1}. Observe that
\begin{align} \label{m111} \nonumber
4\sum_{k, l=1}^N \int_0^{\infty} \int_{\R^N} m^s\overline{\partial_k u_m}\left(\partial_{kl}^2 \varphi_R\right) \partial_l u_m \,dxdm 
&=4\sum_{k, l=1}^{N-1} \int_0^{\infty} \int_{\R^N} m^s\overline{\partial_k u_m}   \left(\partial_{kl}^2 \varphi_R\right) \partial_l u_m \,dxdm  \\ \nonumber
&\quad  +4\sum_{k=1}^{N-1} \int_0^{\infty} \int_{\R^N} m^s\overline{\partial_k u_m}   \left(\partial_{kN}^2 \varphi_R\right) \partial_N u_m \,dxdm \\ \nonumber
& \quad + 4 \sum_{l=1}^{N-1} \int_0^{\infty} \int_{\R^N} m^s\overline{\partial_{N} u_m}   \left(\partial_{Nl}^2 \varphi_R\right) \partial_l u_m \,dxdm \\
& \quad +4 \int_0^{\infty} \int_{\R^N} m^s\overline{\partial_N u_m} \left(\partial_{NN}^2 \varphi_R\right) \partial_N u_m \,dxdm.
\end{align}
Using \eqref{defv} and \eqref{d}, we have that
\begin{align*}
\sum_{k, l=1}^{N-1} \int_0^{\infty} \int_{\R^N} m^s\overline{\partial_k u_m} \left(\partial_{kl}^2 \varphi_R\right) \partial_l u_m \,dxdm &= \sum_{k, l=1}^{N-1} \int_0^{\infty} \int_{\R^N} m^s\overline{\partial_k u_m} \left(\partial_{kl}^2 \psi_R\right) \partial_l u_m \,dxdm\\
&=\int_0^{\infty} \int_{\R^N} m^s \left(\psi_{rr}^2 \psi_R \right) |\nabla_y u_m|^2\,dxdm.
\end{align*}
It is clear to see from \eqref{defv} that $\partial_{jN}^2 \varphi_R=0$ for $1 \leq j \leq N-1$. Therefore, there holds that
$$
\sum_{k=1}^{N-1} \int_0^{\infty} \int_{\R^N} m^s\overline{\partial_k u_m}   \left(\partial_{kN}^2 \varphi_R\right) \partial_N u_m \,dxdm=\sum_{l=1}^{N-1} \int_0^{\infty} \int_{\R^N} m^s\overline{\partial_{N} u_m}   \left(\partial_{Nl}^2 \varphi_R\right) \partial_l u_m \,dxdm=0.
$$
In addition, since $\partial_{NN}^2 \varphi_R=1$, then
\begin{align*}
\int_0^{\infty} \int_{\R^N} m^s\overline{\partial_N u_m} \left(\partial_{NN}^2 \varphi_R\right) \partial_N u_m \,dxdm=\int_0^{\infty}\int_{\R^N} m^s |\partial_{N}u_m|^2 \,dx.
\end{align*}
As an application of Plancherel's identity and Fubini's theorem, we know that
\begin{align} \label{f1}
\begin{split}
\int_0^{\infty} \int_{\R^N} m^s |\nabla u_m|^2\,dxdm&=\frac{\sin \pi s}{\pi}\int_{\R^N} \int_0^{\infty} \frac{m^s}{(|\xi|^2+m)^2} |\xi|^2 |\mathcal{F}(u)|^2 \,dm d\xi\\
&=s\int_{\R^N} |\xi|^{2s}   |\mathcal{F}(u)|^2  \,d\xi=s\int_{\R^N}|(-\Delta)^{s/2} u|^2 \,dx.
\end{split}
\end{align}
Consequently, going back to \eqref{m111} and utilizing \eqref{property}, we get that
\begin{align} \label{v1}
\begin{split}
4\sum_{k, l=1}^N \int_0^{\infty} \int_{\R^N} m^s\overline{\partial_k u_m}   \left(\partial_{kl}^2 \varphi_R\right) \partial_l u_m \,dxdm&=4\int_0^{\infty} \int_{\R^N} m^s |\nabla u_m|^2\,dxdm \\
& \quad -4\int_0^{\infty} \int_{\R^N} m^s\left(1-\partial_{rr}^2 \psi_R \right) |\nabla_y u_m|^2\,dxdm \\
& \leq 4s \int_{\R^N}|(-\Delta)^{s/2} u|^2 \,dx.
\end{split}
\end{align}
Furthermore, applying \cite[Lemma A. 2]{BHL}, we get that
\begin{align} \label{v2}
\int_0^{\infty} \int_{\R^N} m^s\left(\Delta^2 \varphi_R\right) |u_m|^2 \,dxdm \lesssim \|\Delta^2 \varphi_R\|_{\infty}^s \|\Delta \varphi_R\|_{\infty}^{1-s} \|u\|_2^2 \lesssim R^{-2s}.
\end{align}
Now we are going to deal with the second term in the right hand side of \eqref{m1}. Noting \eqref{defv}, one readily finds that
\begin{align} \label{n11}
\begin{split} 
\int_{\R^N}\left(\Delta \varphi_R\right) |u|^{2 \sigma +2} \,dx &=\int_{\R^N}  \left(\Delta \psi_R \right) |u|^{2 \sigma +2} \,dx+\int_{\R^N} |u|^{2 \sigma +2} \,dx\\
&=N \int_{\R^N}  |u|^{2 \sigma +2} \,dx +\int_{\R^N}\left(\Delta \psi_R -N+1\right) |u|^{2 \sigma +2} \,dx.
\end{split}
\end{align}
Obviously, from \eqref{defp}, there holds that $\Delta \psi_R(r) -N+1=0$ for $0 \leq r \leq R$. Therefore, by \eqref{n11}, we conclude that
\begin{align} \label{b1} 
\int_{\R^N}\left(\Delta \varphi_R\right) |u|^{2 \sigma +2} \,dx=N\int_{\R^N}  |u|^{2 \sigma +2} \,dx +\int_{\R}\int_{|y| \geq R}  \left(\Delta \psi_R -N+1\right) |u|^{2 \sigma +2} \,dydx_N.
\end{align}
In what follows, the aim is to estimate the second term in the right hand side of \eqref{b1}. It is simple to notice that
\begin{align} \label{b11}
\int_{\R} \int_{|y| \geq R} |u|^{2 \sigma +2} \,dydx_N  \leq \int_{\R} \|u\|^{2\sigma}_{L^{\infty}(|y| \geq R)} \|u\|_{L^2_y}^{2} \,dx_N.
\end{align}
To estimate the term in the right hand side of \eqref{b11}, we first consider the case that $\sigma=s$. In this case, applying H\"older's inequality, we get that
\begin{align} \label{b1122}
\int_{\R}\int_{|y| \geq R} |u|^{2 s +2} \,dx \leq \left(\int_{\R} \|u\|_{L^{\infty}(|y| \geq R)}^2\, d x_N \right)^{s} \left(\int_{\R} \|u\|^{\frac{2}{1-s}}_{L^2_y} \, d x_N \right)^{1-s}.
\end{align}
In virtue of \eqref{st}, H\"older's inequality and the conservation of mass, we have that
\begin{align} \label{b111}
\begin{split}
\int_{\R}\|u\|^{2}_{L^{\infty}(|y| \geq R)} \, dx_N &\lesssim R^{-(N-2)}\int_{\R} \| (-\Delta_y)^{s/2} u\|_{L^2_y}^{\frac {1}{s}} \| u\|_{L^2_y}^{2-\frac {1}{s}}\,dx_N \\
& \leq R^{-(N-2)} \left(\int_{\R} \| (-\Delta_y)^{s/2} u\|_{L^2_y}^2 \,dx_N \right)^{\frac{1}{2s}} \left(\int_{\R}\| u\|_{L^2_y}^2\,dx_N\right)^{\frac{2s-1}{2s}} \\
& \lesssim R^{-(N-2)} \left(\int_{\R} \| (-\Delta_y)^{s/2} u\|_{L^2_y}^2 \,dx_N \right)^{\frac{1}{2s}}.
\end{split}
\end{align}
Note that $(-\Delta)^{\frac {s}{2}}$ can be equivalently represented as
\begin{align} \label{deff}
(-\Delta)^{\frac {s}{2}} u(x):= C_{N,s} P.V. \int_{\R^N} \frac{u(x)-u(z)}{|x-z|^{N+s}} \,dz,
\end{align}
where $C_{N, s} \in \R$ is a constant given by
$$
C_{N, s}= \left(\int_{\R^N} \frac{1-\cos \xi_1}{|\xi|^{N+s}} \,d\xi\right)^{-1}, \quad \xi=(\xi_1, \xi_2, \cdots, \xi_N).
$$
By \eqref{deff}, we are able to calculate that
\begin{align} \label{fid}
(-\partial_{x_Nx_N}^2)^{\frac {s}{2}} |u|^2=2 |u| (-\partial_{x_Nx_N}^2)^{\frac {s}{2}}|u| -I_s(|u|, |u|),
\end{align}
where
$$
I_{s}(|u|, |u|):=C_{1, s} \int_{\R} \frac{\left(|u|(y, x_N)-|u|(y, x_N')\right)^2}{|x_N-x_N'|^{1+s}} \,dx_N'.
$$
Taking into account \eqref{fid}, Sobolev's inequality, Fubini's inequality, H\"older's inequality and the conservation of mass, we then obtain that
\begin{align} \label{b12} \nonumber
\int_{\R}\|u\|^{\frac{2}{1-s}}_{L^2_y} \, d x_d  & \lesssim  \left(\int_{\R} \left|(-\partial_{x_Nx_N})^{{s}/{2}} \left(\|u\|_{L^2_y}^2\right)\right| \, dx_N\right)^{\frac{1}{1-s}} = \left(\int_{\R} \left| \int_{\R^{N-1}}(-\partial_{x_Nx_N})^{{s}/{2}} \left(|u|^2\right) \, dy\right| \, dx_N\right)^{\frac{1}{1-s}} \\ \nonumber
&\leq \left(\int_{\R^{N-1}} \int_{\R}\left|(-\partial_{x_Nx_N})^{{s}/{2}}\left(|u|^2\right)\right|  \,dx_N dy \right)^{\frac{1}{1-s}}\\ \nonumber
& \lesssim \left(\int_{\R^{N-1}}  \left\|u\right\|_{L^2_{x_N}} \left\|(-\partial_{x_Nx_N})^{{s}/{2}} |u|\right\|_{L^2_{x_N}} \,dy \right)^{\frac{1}{1-s}} + \left(\int_{\R^{N-1}}\left\|(-\partial_{x_Nx_N})^{{s}/{4}} |u| \right\|_{L^2_{x_N}}^2 \,dy \right)^{\frac{1}{1-s}} \\  \nonumber
& \lesssim \left(\int_{\R^{N-1}} \left\|(-\partial_{x_Nx_N})^{{s}/{2}} |u|\right\|_{L^2_{x_N}}^2 \,dy \right)^{\frac{1}{2(1-s)}}  + \left(\int_{\R^{N-1}}\left\|(-\partial_{x_Nx_N})^{{s}/{2}} |u| \right\|_{L^2_{x_N}}\left\|u\right\|_{L^2_{x_N}} \,dy \right)^{\frac{1}{1-s}} \\ 
& \lesssim \left(\int_{\R^{N-1}} \left\|(-\partial_{x_Nx_N})^{{s}/{2}} |u|\right\|_{L^2_{x_N}}^2 \,dy \right)^{\frac{1}{2(1-s)}},
\end{align}
where Sobolev's inequality we used is from the fact that $L^{\frac{1}{1-s}}(\R)$ is continuously embedded into $W^{s,1}(\R)$.
Now going back to \eqref{b1122} and applying \eqref{b111} and \eqref{b12}, we then derive that
\begin{align} \label{bs2}
\int_{\R}\int_{|y| \geq R} |u|^{2 s +2} \,dx \lesssim R^{-(N-2)s}\int_{\R^N} |(-\Delta )^{{s}/{2}} u|^2\,dx.
\end{align}
Next we consider the case that $0<\sigma<s$. In this case, taking into account \eqref{b11} and H\"older's inequality, we know that
\begin{align} \label{b112}
\begin{split}
\int_{\R}\int_{|y| \geq R}  |u|^{2 \sigma +2} \,dx \leq \left(\int_{\R} \|u\|_{L^{\infty}(|y| \geq R)}^{2}\, d x_N \right)^{\sigma} \left(\int_{\R}\|u\|^{\frac{2}{1-\sigma}}_{L^2_y} \, d x_N \right)^{1-\sigma}.
\end{split}
\end{align}
In view of \eqref{fid}, Sobolev's inequality, H\"older's inequality and the conservation of mass, we are able to similarly show that
\begin{align} \label{b121}
\begin{split}
\int_{\R}\|u\|^{\frac{2}{1-\sigma}}_{L^2_y} \, d x_N  & \leq  \left(\int_{\R} \left|(-\partial_{x_Nx_N})^{{\sigma}/{2}} \left(\|u\|_{L^2_y}^2\right)\right| \, dx_N\right)^{\frac{1}{1-\sigma}}\\
&\lesssim \left(\int_{\R^{N-1}} \left\|(-\partial_{x_Nx_N})^{{\sigma}/{2}} |u|\right\|_{L^2_{x_N}}^2 \,dy \right)^{\frac{1}{2(1-\sigma)}} \\
& \lesssim \left(\int_{\R^{N-1}} \left\|(-\partial_{x_Nx_N})^{{s}/{2}} |u|\right\|_{L^2_{x_N}}^{\frac{2\sigma}{s}} \left\|u\right\|_{L^2_{x_N}}^{\frac{2(s-\sigma)}{s}} \,dy \right)^{\frac{1}{2(1-\sigma)}} \\
& \lesssim \left(\int_{\R^{N-1}} \left\|(-\partial_{x_Nx_N})^{{s}/{2}} |u|\right\|_{L^2_{x_N}}^2 \,dy \right)^{\frac{\sigma}{2s(1-\sigma)}}.
\end{split}
\end{align}
Making use of \eqref{b111} and \eqref{b121}, we then obtain from \eqref{b112} that
\begin{align} \label{s2}
\begin{split}
\int_{\R}\int_{|y| \geq R}  |u|^{2 \sigma +2} \,dx & \lesssim R^{-\sigma(N-2)} \left( \int_{\R^N} |(-\Delta )^{{s}/{2}} u|^2\,dx \right)^{\frac{\sigma}{s}} \\
&\lesssim  R^{-\sigma(N-2)} + R^{-\sigma(N-2)} \int_{\R^N} |(-\Delta )^{{s}/{2}} u|^2\,dx.
\end{split}
\end{align}
Combining \eqref{bs2} and \eqref{s2}, we then have that
\begin{align*}
\int_{\R}\int_{|y| \geq R}  |u|^{2 \sigma +2} \,dx \lesssim  R^{-\sigma(N-2)} + R^{-\sigma(N-2)} \int_{\R^N} |(-\Delta )^{{s}/{2}} u|^2\,dx, \quad 0<\sigma \leq s.
\end{align*}
It then follows from \eqref{b1} that
\begin{align*}
\int_{\R^N}\left(\Delta \varphi_R\right) |u|^{2 \sigma +2} \,dx \leq N\int_{\R^N}  |u|^{2 \sigma +2} \,dx + C R^{-\sigma(N-2)} \left(1+\int_{\R^N} |(-\Delta )^{{s}/{2}} u|^2\,dx\right), \quad 0<\sigma \leq s.
\end{align*}
This together with \eqref{v1} and \eqref{v2} then clearly leads the desired conclusion and the proof is completed.
\end{proof}

\begin{proof} [Proof of Theorem \ref{thm1}]
Using Lemma \ref{v} and following the proof of \cite[Theorem 1.1]{BHL} for the case $\sigma>{2s}/{N}$, we are able to conclude the proof.
\end{proof}

\begin{proof}[Proof of Theorem \ref{thm2}]
To prove Theorem \ref{thm2}, we need a refined version of Lemma \ref{v}. Define
$$
\widetilde{\psi}_{1,R}(r):=1-\partial_{rr}^2 \psi_R, \quad  \widetilde{\psi}_{2,R}(r):=N-1-\Delta \psi_R \geq 0.
$$
Taking advantage of \eqref{f1}, \eqref{v1}, \eqref{v2} and \eqref{b1}, we know from \eqref{m1} that
\begin{align} \label{m11}
\begin{split}
\frac{d}{dt} \mathcal{M}_{\varphi_R}[u(t)]=&8s E(u_0)-4\int_0^{\infty} \int_{\R^N} m^s \widetilde{\psi}_{1,R} |\nabla_y u_m|^2\,dxdm \\
& \quad + \frac{4s}{N+2s} \int_{\R}\int_{|y| \geq R}  \widetilde{\psi}_{2,R} |u|^{\frac{4s}{N} +2} \,dydx_N + \mathcal{O}(R^{-2s}).
\end{split}
\end{align}
According to H\"older's inequality, we see that
\begin{align*}
\int_{\R}\int_{|y| \geq R}  \widetilde{\psi}_{2,R} |u|^{\frac{4s}{N} +2} \,dydx_N &\leq \int_{\R} \left\| \widetilde{\psi}_{2,R}^{\frac{N}{2N+4s}}|u|\right\|^{\frac{4s}{N}}_{L^{\infty}(|y| \geq R)} \left\|\widetilde{\psi}_{2,R}^{\frac{N}{2N+4s}}|u|\right\|_{L^2_y}^{2} \,dx_N \\
& \leq \left(\int_{\R} \left\|\widetilde{\psi}_{2,R}^{\frac{N}{2N+4s}}|u|\right\|^2_{L^{\infty}(|y| \geq R)} \,dx_N\right)^{\frac{2s}{N}} \left(\int_{\R}\left\|\widetilde{\psi}_{2,R}^{\frac{N}{2N+4s}}|u|\right\|_{L^2_y}^{\frac{2N}{N-2s}} \,dx_N\right)^{\frac{N-2s}{N}}.
\end{align*}
It follows from \eqref{b111} and \eqref{b121} with $\sigma=2s/N$ that
\begin{align*}
\int_{\R} \left\|\widetilde{\psi}_{2,R}^{\frac{N}{2N+4s}}|u|\right\|^2_{L^{\infty}(|y| \geq R)} \,dx_N \lesssim R^{-(N-2)} \left(\int_{\R} \left\| (-\Delta_y)^{s/2} \left(\widetilde{\psi}_{2,R}^{\frac{N}{2N+4s}}|u|\right)\right\|_{L^2_y}^2 \,dx_N \right)^{\frac{1}{2s}},
\end{align*}
\begin{align*}
\int_{\R}\left\|\widetilde{\psi}_{2,R}^{\frac{N}{2N+4s}}|u|\right\|_{L^2_y}^{\frac{2N}{N-2s}} \,dx_N \lesssim \left(\int_{\R^{N-1}} \left\|(-\partial_{x_Nx_N})^{{s}/{2}} \left(\widetilde{\psi}_{2,R}^{\frac{N}{2N+4s}}|u|\right)\right\|_{L^2_{x_N}}^2 \,dy \right)^{\frac{1}{N-2s}}.
\end{align*}
Consequently, we have that
\begin{align} \label{mcb}
\begin{split}
\int_{\R}\int_{|y| \geq R} \widetilde{\psi}_{2,R} |u|^{\frac{4s}{N} +2} \,dydx_N &\lesssim R^{-\frac{2s(N-2)}{N}}\left(\int_{\R^N} \left|(-\Delta )^{{s}/{2}}\left(\widetilde{\psi}_{2,R}^{\frac{N}{2N+4s}}u\right)\right|^2\,dx\right)^{\frac 1N} \\
& \leq \eta \int_{\R^N} \left|(-\Delta )^{{s}/{2}}\left(\widetilde{\psi}_{2,R}^{\frac{N}{2N+4s}}u\right)\right|^2\,dx+ \mathcal{O} \left(\eta^{-\frac{1}{N-1}} R^{-\frac{2s(N-2)}{N-1}}\right),
\end{split}
\end{align}
where we also used Young's inequality with $\eta>0$. Moreover, adapting the elements presented in the proof of \cite[Lemma 2.3]{BHL}, we are able to derive that
\begin{align*}
s\int_{\R^N} \left|(-\Delta )^{{s}/{2}}\left(\widetilde{\psi}_{2,R}^{\frac{N}{2N+4s}}u\right)\right|^2\,dx=\int_0^{\infty} \int_{\R^N} m^s \widetilde{\psi}_{2,R}^{\frac{N}{N+2s}}|\nabla u_m|^2 \,dxdm+ \mathcal{O} \left(1+R^{-2}+R^{-4}\right).
\end{align*}
It then yields from \eqref{mcb} that
\begin{align*}
\int_{\R}\int_{|y| \geq R} \widetilde{\psi}_{2,R} |u|^{\frac{4s}{N} +2} \,dydx_N & \leq \frac{\eta}{s} \int_0^{\infty} \int_{\R^N} m^s\widetilde{\psi}_{2,R}^{\frac{N}{N+2s}}|\nabla u_m|^2 \,dxdm \\
& \quad +\mathcal{O} \left(\eta^{-\frac{1}{N-1}} R^{-\frac{2s(N-2)}{N-1}}+\eta\left(1+R^{-2}+R^{-4}\right)\right).
\end{align*}
Inserting into \eqref{m11}, we then conclude that
\begin{align} \label{rv}
\begin{split}
\frac{d}{dt} \mathcal{M}_{\varphi_R}[u(t)] &=8s E(u_0)-4\int_0^{\infty} \int_{\R^N} m^s \left(\widetilde{\psi}_{1,R}- \frac{\eta}{N+2s}\widetilde{\psi}_{2,R}^{\frac{N}{N+2s}}\right) \widetilde{\varphi}_R |\nabla_y u_m|^2\,dxdm \\
& \quad +\mathcal{O} \left(\eta^{-\frac{1}{N-1}} R^{-\frac{2s(N-2)}{N-1}}+R^{-2s}+\eta\left(1+R^{-2}+R^{-4}\right)\right).
\end{split}
\end{align}
At this point, using the refined version of Lemma \ref{v} given by \eqref{rv} and following the proof of \cite[Theorem 1.1]{BHL} for the case $\sigma={2s}/{N}$, we are able to conclude the proof. This completes the proof.
\end{proof}

To discuss blow-up of solutions to \eqref{equ} with initial data belonging to $\Sigma$. We shall introduce a new localized virial quantity. Let $\psi_R$ be defined by \eqref{defp}. The localized virial quantity is indeed defined by
\begin{align} \label{defm1}
\mathcal{M}_{\psi_R}[u]:=2 \textnormal{Im} \int_{\R^N} \overline{u} \left(\nabla \psi_R \cdot \nabla_y u \right) \,dx.
\end{align}

\begin{lem} \label{v1}
Let $N \geq 3$, $1/2<s<1$ and $0<\sigma \leq s$. Suppose that $u \in C([0, T); H^s(\R^N))$ is the solution to \eqref{equ} with initial datum $u_0 \in \Sigma$. Then, for any $t \in [0, T)$, there holds that
\begin{align} \label{rv1}
\begin{split}
\frac{d}{dt}\mathcal{M}_{\psi_R}[u(t)] &\leq 4s \int_{\R^N}|(-\Delta)^{s/2} u(t)|^2 \,dx -\frac{2 \sigma (N-1)}{\sigma+1} \int_{\R^N}  |u(t)|^{2 \sigma +2} \,dx  \\
& \quad + C \left( R^{-2s} +\left(1+R^{-\sigma(N-2)} \int_{\R^N} |(-\Delta )^{{s}/{2}} u(t)|^2\,dx\right) \right)\\
&=4\sigma (N-1) E(u_0)-2(\sigma (N-1) -2s)\int_{\R^N} |(-\Delta )^{{s}/{2}} u(t)|^2\,dx \\
& \quad + C \left( R^{-2s} +\left(1+R^{-\sigma(N-2)} \int_{\R^N} |(-\Delta )^{{s}/{2}} u(t)|^2\,dx\right) \right).
\end{split}
\end{align}
\end{lem}
\begin{proof}
Replacing the roles of $\varphi_R$ in the proof of Lemma \ref{v} by $\psi_R$ and repeating the proof of Lemma \ref{v}, we then obtain the desirable conclusion. This completes the proof.
\end{proof}

\begin{proof} [Proof of Theorem \ref{thm11}]
Since $E[u_0]<0$ and $\sigma(N-1)>2s$, by applying Lemma \ref{v1}, then we are able to get that, for $R>0$ large enough,
$$
\frac{d}{dt}\mathcal{M}_{\psi_R}[u(t)] \leq 2\sigma (N-1) E(u_0)<0.
$$
This then immediately implies the desirable conclusion and the proof is completed.
\end{proof}

\section*{Statements and Declarations} 
\subsection*{Conflict of interests}The authors declare that there are no conflict of interests.
\subsection*{Acknowledgements} T. Gou was supported by the National Natural Science Foundation of China (No. 12101483) and the Postdoctoral Science Foundation of China. The research of V.D. R\u adulescu was supported by the grant ``Nonlinear Differential Systems in Applied Sciences" of the Romanian Ministry of Research, Innovation and Digitization, within PNRR-III-C9-2022-I8/22. Z. Zhang was supported by the National Key R $\&$ D Program of China (2022YFA1005601) and the National Natural Science Foundation of China (No. 12031015). T. Gou warmly thanks Dr. Luigi Forcella for fruitful discussions on the proof of Theorem \ref{thm11}.
\

\end{document}